\documentclass[11pt]{amsart}
\usepackage{amssymb}
\usepackage{graphicx,tikz}
\usetikzlibrary{arrows}
\newtheorem{theorem}{Theorem}[section]

\newtheorem{proposition}[theorem]{Proposition}
\newtheorem{lemma}[theorem]{Lemma}
\newtheorem{fact}[theorem]{Fact}

\newtheorem{corollary}[theorem]{Corollary}
\newtheorem{definition}[theorem]{Definition}

\newcommand\Z{\mathbb Z}

\newcommand\N{\mathbb N}

\newcommand\W{{\mathcal W}_{m,k}}

\title{Free limits of Thompson's group $F$}

\author[Azer Akhmedov]{Azer Akhmedov}
\address{Department of Mathematics, North Dakota State University}
\email{azer.akhmedov@ndsu.edu}

\author[Melanie Stein] {Melanie Stein}
\address{Department of Mathematics, Trinity College, Hartford, CT 06106}
\email{melanie.stein@trincoll.edu}

\author[Jennifer Taback] {Jennifer Taback}
\address{Department of Mathematics, Bowdoin College, Brunswick, ME 04011}
\email{jtaback@bowdoin.edu}

\thanks{The third author acknowledges partial support from National Science Foundation Grant DMS-0604645. The second and third authors acknowledge partial support from a Bowdoin College Faculty Research grant.
The authors would like to thank Matt Brin, Ben Fine, David Fisher, Olga Kharlampovich, Alexei Myasnikov and Roland Zarzycki for helpful conversations during the writing of this paper.}


\date{\today}

\begin{document}

\maketitle

\begin{abstract}
We produce a sequence of markings $S_k$ of Thompson's group $F$ within the space ${\mathcal G}_n$ of all marked $n$-generator groups so that the sequence $(F,S_k)$ converges to the free group on $n$ generators, for $n \geq 3$.
In addition, we give presentations for the limits of some other natural (convergent) sequences of markings to consider on $F$ within ${\mathcal G}_3$, including $(F,\{x_0,x_1,x_n\})$ and $(F,\{x_0,x_1,x_0^n\})$.
\end{abstract}

\section{Introduction}

Sela defined the notion of a limit group in conjunction with his
solution to the problem of Tarski which asks whether all free
groups of rank at least $2$ have the same elementary theory
\cite{S1,S2}.  Limit groups arise in Sela's analysis of equations
in free groups, and he shows that they coincide with the class of
finitely generated, fully residually free groups [Se1].  Work of
Sela, along with Kharlampovich and Myasnikov \cite{S1,KM1,KM2}
shows that limit groups can be constructed recursively from
building blocks consisting of free, surface and free abelian
groups by taking a finite sequence of free products and
amalgamations over $\Z$.  Alternately, a group is a limit group in the sense of Sela
if and only if it is an iterated generalized double, defined below in Section \ref{sec:notfree}.

Define a marked group $(G,S)$ to be a group $G$ with a fixed and
ordered set of generators $S=\{g_1,g_2, \cdots ,g_n\}$, and let
${\mathcal G}_n$ be the set of all groups marked by $n$ elements
up to isomorphism of marked groups.  A marked group $(G,S)$ is
equipped with a canonical epimorphism from the free group on $|S|$
letters to $G$.  The space ${\mathcal G}_n$ admits a topology
in which two marked groups $(G_1,S_1)$ and $(G_2,S_2)$ are at
distance at most $e^{-R}$ if they have same relations of length at
most $R$. With respect to this topology, the limit groups of Sela,
equivalently the class of all finitely generated fully residually
free groups, arise naturally as limits of marked free groups.

This topological approach towards marked groups opens the
notion of limit groups to include limits of other, non free,
groups within ${\mathcal G}_n$, for a fixed $n$.  This topology
was defined in \cite{G} by Grigorchuk, and an earlier equivalent
construction was presented in \cite{C}; Champetier and Guirardel
study this topology in \cite{CG}, and Guyot and Stalder in
\cite{St,GS} investigate limits of marked copies of
Baumslag-Solitar groups.

Just as the class of finitely generated fully residually free
groups arise naturally as limits of marked free groups, one can
extend the definition of \lq \lq fully residually" to non-free
classes of groups in such a way that these groups arise naturally
as limits of marked sequences of other, non-free, groups. More
specifically, a finitely generated group $G$ is defined to be {\em
fully residually ${\mathcal P}$} (where we view ${\mathcal P}$ as
a property of groups, e.g. free or finite) if for any finite
collection $\{w_1,w_2, \cdots ,w_k\}$ of elements of $G$, there is
a surjective homomorphism $\phi$ from $G$ to a ${\mathcal P}$
group so that the images $\{\phi(w_1),\phi(w_2), \cdots
,\phi(w_k)\}$ of the original set of elements are all nontrivial.
In the case ${\mathcal P}$=free, we can omit the requirement that
$\phi$ be surjective, as any subgroup of a free group is itself
free.  Just as for the special case of fully residually free
groups, if a group $G$ is fully residually ${\mathcal P}$, there is a
sequence of markings of the ${\mathcal P}$ group so that the
sequence of marked groups converges to $G$. With this definition,
any group which is fully residually Thompson arises as a limit of
marked copies of Thompson's group $F$.  One aim of this paper is
to show that the free group $F_k$ for $k \geq 3$ is fully
residually Thompson.

The goal of this paper is to analyze several sequences of markings
of Thompson's group $F$.  We begin in Section \ref{sec:free} by partially
answering a question of Sapir, and show in Corollary
\ref{cor:free-limit} that there exist sequences of marked copies
of Thompson's group $F$ which converge to the free group $F_k$
within ${\mathcal G}_k$ for any $k \geq 3$, that is, we show that
the free group is fully residually Thompson.  Brin \cite{B} has
recently shown that there is a sequence of markings of $F$ in
${\mathcal G}_2$ which converges to the free group $F_2$.
This is related to the notion of a group exhibiting {\em $k$-free-like behavior},
defined by Olshanskii and Sapir  in \cite{OS}.   A group $G$ is said to be $k$-free-like for $k \geq 2$
if there exists a sequence of generating sets $Z_i$ for $i \geq
1$, each with $k$ elements, such that the Cayley graph
$\Gamma(G,Z_i)$ satisfies no relation of length less than $i$, and
the Cheeger constant of this graph is uniformly (in $i$) bounded
away from zero.  The Cheeger constant is the infimum over all
subsets $A$ of the group of the ratio of the size of the boundary
of $A$ to the size of $A$, with respect to a fixed generating set.
An answer to the question of whether $F$ is amenable will also decide
whether $F$ exhibits $k$-free-like behavior.

In Section \ref{sec:notfree} we investigate the limits of several
sequences of marked copies of $F$ within ${\mathcal G}_3$ where
the markings are chosen to be very ``natural", for example the
sequence $\{x_0,x_1,x_n\}$, where the generators are taken from
the standard infinite presentation for $F$.  We generalize our
concrete examples to sequences of markings where the third element
in the triple is simply subject to certain conditions on its
support.  We note that in each case, there is an amalgamated product (or HNN extension) of copies
of Thompson's group $F$, or subgroups of $F$, over a maximal abelian subgroup which is a generalized double over the limit group obtained.

The problem of determining all sequences of
markings of the form $\{x_0,x_1,g_n\}$ of $F$, where $g_n\in F$
which converge in ${\mathcal G}_3$, and the presentation of any
resulting limit groups, is extremely interesting to the authors.
In his thesis, Zarzycki considers this problem as well; he
has obtained preliminary results which state that limit of a
sequence of marked copies of $F$ using markings of the form
$\{x_0,x_1,g_n\}$ can never be a central HNN-extension  \cite{Z1,Z2}.  His
methods are significantly different from the techniques presented
here.

\section{Preliminaries}

In this section we present brief background material on several topics used in this paper.

\subsection{A brief introduction to Thompson's group $F$}
\label{sec:introF}

Thompson's group $F$ can be viewed as the group of piecewise-linear orientation-preserving homeomorphisms of the unit interval, subject to two conditions:
\begin{enumerate}
\item the coordinates of all breakpoints lie in the set of dyadic rational numbers, and
\item the slopes of all linear pieces are powers of $2$.
\end{enumerate}
Group elements can be viewed uniquely in this way if we require that the slope change at all given breakpoints.  Group multiplication then simply corresponds to function composition.

This group is commonly studied via a standard infinite presentation ${\mathcal P}$:
$${\mathcal P} = \langle x_0,x_1,x_2, \cdots  | x_i^{-1} x_j x_i = x_{j+1} \text{ for } i < j \rangle.$$
It is clear from the above presentation that $x_0$ and $x_1$ are sufficient to generate the group, and we thus obtain the standard finite presentation ${\mathcal F}$:
$${\mathcal F} = \langle x_0,x_1 | [x_1x_0^{-1},x_0^{-1}x_1x_0], \ [x_1x_0^{-1},x_0^{-2}x_1x_0^2] \rangle.$$
The generators $x_0, \ x_1$ and $x_n$ are depicted as homeomorphisms of the interval in Figure \ref{fig:generators}. Notice that the support of $x_n$ is exactly the interval $[1-\frac{1}{2^n},1]$.  As a consequence, any element with support contained in $[0, 1-\frac{1}{2^n}]$ will commute with $x_n$, as elements of $F$ with disjoint support always commute.  Analyzing the relators in the finite presentation ${\mathcal F}$, it is not hard to see that the support of $x_1x_0^{-1}$ is $[0, \frac{3}{4}]$, while the supports of $x_2 = x_0^{-1}x_1x_0$ and $x_3 = x_0^{-2}x_1x_0^2$ are contained in $[\frac{3}{4},1]$, hence these elements commute.

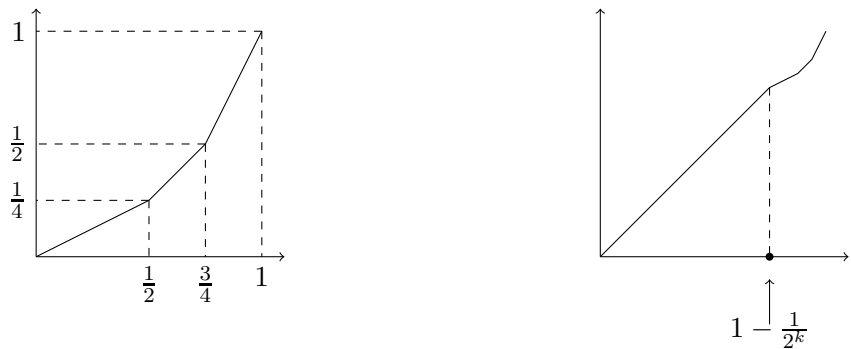
\begin{figure}[h!]
\begin{center}
\begin{tikzpicture}[scale=3]
\draw[->] (0,0) -- (1.1,0);
\draw[->] (0,0) -- (0,1.1);
\draw (0,0)--(.5,.25)--(.75,.5)--(1,1);
\draw[dashed] (.5,.25)--(.5,0);
\draw (.5,0) node [anchor=north]{$\frac{1}{2}$};
\draw[dashed] (.75,.5)--(.75,0);
\draw (.75,0) node [anchor=north]{$\frac{3}{4}$};
\draw[dashed] (1,1)--(1,0);
\draw (1,0) node [anchor=north]{1};

\draw[dashed] (.5,.25)--(0,.25);
\draw (0,.25) node [anchor=east]{$\frac{1}{4}$};
\draw[dashed] (.75,.5)--(0,.5);
\draw (0,.5) node [anchor=east]{$\frac{1}{2}$};
\draw[dashed] (1,1)--(0,1);
\draw (0,1) node [anchor=east]{1};

\draw[->] (2.5,0) -- (3.6,0);
\draw[->] (2.5,0) -- (2.5,1.1);
\draw (2.5,0)--(3,.5)--(3.25,.75)--(3.375,.8125)--(3.4375,.875)--(3.5,1);
\draw[dashed] (3.25,.75)--(3.25,0);
%\draw[dashed] (3.25,.75)--(2.5,.75);
\fill (3.25,0) circle (.5pt);
%\fill (2.5,.75) circle (.5pt);
\draw[->] (3.25,-.3)--(3.25,-.1);
%\draw[dashed] (3.75,.75)--(4,.75)--(4,1)--(3.75,1)--(3.75,.75);
\draw (3.25,-.2) node [anchor=north] {$1-\frac{1}{2^k}$};
\end{tikzpicture}
\caption{The generators $x_0$ and $x_n$ of $F$ as homeomorphisms of $[0,1]$.} \label{fig:generators}
\end{center}
\end{figure}

With respect to the infinite presentation ${\mathcal P}$, elements of $F$ have a standard (infinite) normal form, given by $$x_{i_1}^{r_1}
x_{i_2}^{r_2}\ldots x_{i_k}^{r_k} x_{j_l}^{-s_l} \ldots
x_{j_2}^{-s_2} x_{j_1}^{-s_1} $$ where $r_i, s_i >0$, $0 \leq i_1<i_2
\ldots < i_k$ and $0 \leq j_1<j_2 \ldots < j_l$, and if
both $x_i$ and $x_i^{-1}$ occur, so does $x_{i+1}$ or  $x_{i+1}^{-1}$, as discussed by
Brown and Geoghegan in \cite{BG}.   A {\em positive} word (resp. {\em negative word}) has normal form containing only positive (resp. negative) exponents.

For a thorough introduction to Thompson's group $F$ we refer the reader to \cite{CFP}.

\subsection{Girth and limit groups}

The girth of a group $G$ with respect to a finite generating set
$S$ is defined to be the length of the shortest relator satisfied
in $(G,S)$.  The girth of a finitely generated group $G$ is the
supremum of $girth(G,S)$ over all finite generating sets $S$ for
$G$.  It is proven in \cite{Ak} that a non-cyclic, finitely
generated hyperbolic group (or one-relator or linear group) has
infinite girth if and only if it is not virtually solvable.

We will use the notion of girth together with the following proposition of Stalder, which provides
 an algebraic formulation of convergence of a series of marked groups in this topology, to exhibit free limits of marked copies of Thompson's group $F$.
If $(G,S)$ is a marked group and $|S| = k$, we say that $S$ is a marking of $G$ of {\em length} $k$.

\begin{proposition}[\lbrack S\rbrack]\label{prop:stalder}
Let $(\Gamma_n,S_n)$ be a sequence of marked groups on $k$ generators.  The following are equivalent:
\begin{enumerate}\item $(\Gamma_n,S_n)$ is convergent in ${\mathcal G}_k$;
\item for all $w \in F_k$ we have either $w=1$ in $\Gamma_n$ for $n$ large enough, or $w \neq 1$ in $\Gamma_n$ for $n$ large enough, where $F_k$ is the free group on $k$ letters.
\end{enumerate}
\end{proposition}

In the sections below, we will be interested not only in whether particular
sequences of marked groups converge, but in the presentation of the limit group of such a
convergent subsequence.  The following proposition, while a restatement of the definition
of convergence of a sequence of marked groups, states explicitly how we characterize this limit group.

\begin{proposition}\label{convergence to anything}
Let $(\Gamma_n,S_n)$ be a sequence of marked groups on $k$
generators. Then $(\Gamma_n,S_n)$ converges to $(\Gamma,S)$ in
${\mathcal G}_k$ if for all words $w \in F_k$
\begin{enumerate}
\item if $w=1$ in $\Gamma$, then $w=1$ in $\Gamma_n$ for
sufficiently large $n$. \item if $w =1$ in $\Gamma_n$ for
infinitely many $n$, then $w=1$ in $\Gamma$.
\end{enumerate}
\end{proposition}

In the case that the limit group is a free group, the first
property holds trivially, so we need only check the second.
Combining the notion of girth with the above proposition yields a
straightforward characterization of when a sequence of markings of
a fixed group converges to a free group. Namely:

\begin{proposition}\label{covergence to free}
 If a group $G$ has a sequence of markings $S_n$ of length
$k$ so that the girth of $(G,S_n)$ goes to infinity with $n$, then
the sequence $(G,S_n)$ converges to $F_k$ in ${\mathcal G}_k$.
\end{proposition}

\section{Free limits of Thompson's group $F$}
\label{sec:free}

The goal of this section is to prove that there is a sequence
$\{S_n\}$ of markings of $F$ so that $(F,S_n)$ converges to the
free group of rank $k$ in ${\mathcal G}_k$, for all $k \geq 3$.
This is proven by exhibiting a sequence of markings of $F$ of a fixed length so that
the girth of the group with respect to these markings approaches infinity.  As a corollary we obtain that the same sequence of \
markings on $F$ converges to the free group of the appropriate rank.
As the finiteness of the girth of a group is closely related to whether the group satisfies a law, we first recall
the result of Brin and Squier \cite{BS1} (and reproven by Ab$\acute{e}$rt \cite{A} and later by Esyp \cite{E}) that Thompson's group $F$ satisfies no law.

First define $\W$ to be the set of all nontrivial reduced words in the free group $F(a,b_1,b_2, \cdots ,b_{k-1})$ of rank $k$ of length at most $m$.
The reason for using both $a$ and $b_j$ to denote the generators of the free group will be clear below.

\begin{proposition} \label{prop:length-at-most-m}
Fix $m \in \N$ and let $k \geq 2$.  There exist $u_1, u_2, \cdots u_{k} \in F$ so that for any $w \in \W$ the word $w(u_1, u_2, \cdots u_{k}) \in F$ is nontrivial.
\end{proposition}

\begin{proof}
Suppose that there are $l$ words $w_1, \cdots ,w_l$ in $\W$.  As
$F$ satisfies no group law, for each $i$ with $1 \leq i \leq l$ we
can find $u'_{i,1}, u'_{i,2}, \cdots u'_{i,k} \in F$ so that
$w_i(u'_{i,1}, u'_{i,2}, \cdots u'_{i,k} )$ is nontrivial. For any
two dyadic rationals $a$ and $b$ with $a<b$, there is an
isomorphism $\phi_{[a,b]}$ from $F$ to the isomorphic copy of $F$
supported on the interval $[a,b]$, which we denote
$F_{[a,b]}$, as detailed in \cite{BiSt}. Choosing dyadic rationals $0=a_0<a_1<a_2 \cdots
<a_{i-1}<a_i=1$, let $I_j=[a_{j-1},a_j]$ for $1 \leq j \leq
i$ and define $\phi_{I_j}:F \rightarrow F_{I_j}$ to be the corresponding
isomorphisms.  Let $u_{j,r} = \phi_{I_j}(u'_{j,r})$ for $1 \leq r
\leq k$ and $1 \leq j \leq l$. It is clear that the support of
$u_{j,r}$ lies in the interval $I_j$.

For $1 \leq r \leq k$, define $u_r(x) = u_{j,r}(x)$ for $x \in
I_j$.  If $w \in \W$, then $w(u_1, u_2, \cdots u_{k})$ must be
nontrivial on at least one interval $I_j$ by construction.
\end{proof}

The next proposition shows that for any word $w=w(a,b_1,b_2,
\cdots b_{k-1}) \in \W$, we can always find $u_1,u_2, \cdots
u_{k-1} \in F$ so that $w(x_0,u_1,u_2, \cdots ,u_{k-1})$ is
nontrivial. We first state a basic lemma regarding the
construction of elements of $F$ with certain proscribed values.  A
proof of this lemma can be found in \cite{CFP}.

\begin{lemma}[\lbrack CFP\rbrack, Lemma 4.2]\label{lemma:dyadic-list}
If $0=x_0 < x_1 < x_2 < \cdots < x_n = 1$ and $0=y_0 < y_1 < y_2 <
\cdots < y_n = 1$ are partitions of $[0,1]$ consisting of dyadic
rational numbers,  then there exists $f \in F$ so that $f(x_i) =
y_i$ for $i=0,1,2, \cdots ,n$.  Furthermore, if $x_{i-1} =
y_{i-1}$ and $x_i = y_i$ for some $i$ with $1 \leq i \leq n$, then
$f$ can be taken to be the identity on the interval
$[x_{i-1},x_i]$.
\end{lemma}

\begin{proposition}
\label{prop:x0-u-substitution} Fix $m \in \N$, $0<\epsilon < \frac{1}{2}$, and $k \geq 2$.
For any $w=w(a,b_1,b_2, \cdots b_{k-1}) \in \W$, there exist elements
$u_1,u_2, \cdots u_{k-1} \in F$ so that for $1 \leq i \leq k-1$, we have $Support(u_i) \subseteq [0, \epsilon )$ and $w(x_0,u_1,u_2, \cdots
,u_{k-1})|_{[0, \epsilon)}\neq 1$.
\end{proposition}

\begin{proof}
As $m$ is fixed, define ${\mathcal N}_m = \{-m,-(m-1), \cdots
,0,1,2, \cdots ,m\}$.  Choose $I_0 \subset (0,\epsilon)$ to be
a small closed interval whose endpoints lie in $\Z[\frac{1}{2}]$
with the property that the collection of intervals
$\{I_i=x_0^i(I_0)\}$ is pairwise disjoint for $i \in {\mathcal
N}_m$, and all $I_i \subset (0,\epsilon)$. Given $w \in \W$ we
will construct $u_1,u_2, \cdots u_{k-1} \in F$ supported in
$\cup_{i \in {\mathcal N}_m} I_i$ so that $w(x_0,u_1,u_2, \cdots
u_{k-1})$ is nontrivial.

More specifically, choosing $x$ in the interior of $I_0$, we will
define a sequence of elements $u_1,u_2, \cdots u_{k-1}$ in $F$ so
that $w(x_0,u_1,u_2, \cdots ,u_{k-1})(x) \neq x$.  Each element
$u_j$ will be constructed so that
\begin{enumerate}
\item $u_j$ preserves each interval, though not point-wise, that is, $u_j(I_i) = I_i$ for $i \in {\mathcal N}_m$, and
\item $u_j$ is the identity off the union of these intervals, that is, for any $y \notin \cup_{i \in {\mathcal N}_m} I_i$, we have $u_j(y)=y$.
\end{enumerate}
The construction will be accomplished by selecting, for each $j$
and each $i$, a (possibly empty) pair of sequences of points
$\alpha_1 < \alpha_2 < \cdots \alpha_r$ and $\beta_1 < \beta_2 <
\cdots < \beta_r$ in the interior of $I_i$, and applying Lemma
\ref{lemma:dyadic-list} to define $u_j|_{I_i}$.

To begin the construction, choose $x \in {\mathbb Z}[\frac{1}{2}]$
in the interior of $I_0$. Let $$w=a^{e_1}B_1 a^{e_2} B_2 \cdots
a^{e_r}B_r$$ be a word in $F(a,b_1,b_2, \cdots ,b_{k-1})$ where
each $B_j$ is a word in $b_1^{\pm 1},b_2^{\pm 1}, \cdots
,b_{k-1}^{\pm 1}$. We allow $e_1 = 0$ and $B_r = 1$, but $e_l \neq
0$ for $1<l\leq r$ and $B_l$ is not the trivial word for $1 \leq l
 < r$.

For $1 \leq s \leq r$, let $B_s=b_{s,1}b_{s,2} \dots b_{s,q_s}$ for $b_{s,j} \in
\{b_1^{\pm1}, \dots ,b_{k-1}^{\pm 1} \}$. All points that are
chosen in the construction below are assumed to lie both in the
interior of an interval $I_j$ and in $\Z[\frac{1}{2}]$. First
choose an increasing sequence of dyadic rationals in the interior
of $I_0$ beginning with $x$ (which we relabel in keeping with our
indexing scheme),
$$x=y_r=y_r^{q_r+1}< y_r^{q_r} < y_r^{q_r-1} < \cdots <
y_r^1=x_r.$$ Note that the length of this sequence is one more
than the length of $B_r$.  Moving through the word $w$ from right
to left, we now consider $a^{e_r}$, which indicates
which interval we use to choose the next sequence of points. More precisely, recalling that $x_0$ will be substituted for $a$, this next sequence will be chosen in the interior of
$x_0^{e_r}(I_0)$, beginning with the image of the final point in the
previous sequence under $x_0^{e_r}$ and with length $|B_{r-1}|+1$.
Namely, let $y_{r-1}=x_0^{e_r}(x_r) \in x_0^{e_r}(I_0)=I_j$ for $I_j \neq
I_0$, and then choose an increasing sequence of dyadic rationals in
the interior of $I_j$,
$$y_{r-1}=y_{r-1}^{q_{r-1}+1} < y_{r-1}^{q_{r-1}} <
y_{r-1}^{q_{r-1}-1}< \cdots < y_{r-1}^1=x_{r-1}.$$

Continue in this way through the word $w$, constructing increasing sequences of points in
the various intervals $I_j$. We remark that for each index $s$ the sequence constructed consists of
more than one point, with the possible exception of $s=r$, that is, the initial sequence constructed.
When $s \neq r$, we know that $B_s \neq 1$, hence $y_s < x_s$.
Notice that it is possible for more than one such sequence to be chosen
within a single interval $I_j$. If it is the case that two sequences of the
form $\{ y_l^i \}$ and $\{y_h^i\}$ for fixed $l$ and $h$ where
$l<h$ both are chosen in $I_j$, then by construction, $y_h^s \leq y_l^p$ for any
superscripts $s$ and $p$. In fact, we claim this inequality is
strict. To see this, note that since $h >1$, $e_h \neq 0$, and
hence $h-1 \neq l$. In particular, $l+1 \neq r$, and hence
$y_{l+1} < x_{l+1}$. Therefore, $y_h^s \leq y_{l+1} < x_{l+1} \leq
y_l^p$, and so $y_h^s < y_l^p$, as desired.

Now for each pair of points $y_s^{i+1}<y_s^i$, if $b_{s,i}=b_n^{\pm1}$, we define an ordered pair of points $(z, u_n(z))$ as follows.
If $b_{s,i}=b_n$, choose $(z, u_n(z))=(y_s^{i+1},y_s^i)$, and if $b_{s,i}=b_n^{-1}$, choose $(z, u_n(z))=(y_s^i, y_s^{i+1})$.
We claim that for any $n$, the collection of pairs $\{(z_{\alpha},
u_n(z_{\alpha}))\}$ defined above satisfy the hypotheses of Lemma
\ref{lemma:dyadic-list}. Namely, the domain points
$\{z_{\alpha}\}$ are all distinct, and if $z_{\alpha_1} <
z_{\alpha_2}$ then $u_n(z_{\alpha_1}) < u_n(z_{\alpha_2})$.

To verify this claim, we must show that if
$(z_1,u_n(z_1))$ and $(z_2,u_n(z_2))$ are two such pairs,
corresponding to two letters $b_n^{\epsilon_1}$ and
$b_n^{\epsilon_2}$ in the word $w$, for some $n \in\{1, 2, \ldots,
k-1\}$ and $\epsilon_1,\epsilon_2 \in \{+1,-1\}$ where $z_1 \leq
z_2$, then in fact $z_1 \neq z_2$ and $u_n(z_1) < u_n(z_2)$.

To see this, first consider the case that both $b_n^{\epsilon_1}$
and $b_n^{\epsilon_2}$ occur within the same subword $B_j$. If
$\epsilon_1=\epsilon_2$, the claim is clear. If $\epsilon_1 \neq
\epsilon_2$, since $B_j$ is freely reduced, $b_n^{\epsilon_1}$ and
$b_n^{\epsilon_2}$ are not adjacent in the word. Hence,
$b_n^{\epsilon_1}=b_{j,i}$ and $b_n^{\epsilon_2}=b_{j,l}$ with
$l+1<i$, and $y_j^{i+1}<y_j^i<y_j^{l+1}<y_j^l$. Hence, regardless
of the values of $\epsilon_1$ and $\epsilon_2$, since we have the
setwise equalities $\{z_1, u_n(z_1)\}=\{y_j^{i+1},y_j^{i}\}$ and
$\{z_2, u_n(z_2)\}=\{y_j^{l+1},y_j^{l}\}$, we see that $z_1 \neq
z_2$ and $u_n(z_1)<u_n(z_2)$.

On the other hand, suppose $b_n^{\epsilon_1}$
is in the subword $B_j$ and $b_n^{\epsilon_2}$ is in the subword
$B_k$ with $k\neq j$. Then if $z_1\in I_i$ and $z_2 \in I_l$ with
$i \neq l$, since $I_l \cap I_i=\emptyset$, the claim is clearly
true. So suppose both $z_1$ and $z_2$ are in $I_i$. Then since
$z_1 \leq z_2$, $k<j$. Then $\{z_1,
u_n(z_1)\}=\{y_j^{r},y_j^{r+1}\}$ and $\{z_2,
u_n(z_2)\}=\{y_k^{s},y_k^{s+1}\}$ for some superscripts $r$ and
$s$. Then as previously established, $z_1<z_2$ and
$u_n(z_1)<u_n(z_2)$.

We have shown that for a given $n \in \{1, 2, \dots,
k-1\}$, the collection of all of the pairs $\{(z_{\alpha},
u_n(z_{\alpha}))\}$ defined above satisfy the hypothesis Lemma
\ref{lemma:dyadic-list}.
Therefore, we can apply Lemma \ref{lemma:dyadic-list} for any pair
$j$ and $n$ to define $u_n|_{I_j}$. If, for a given pair, no
points have been chosen in $I_j$ to be domain and range points for
$u_n$, then simply define $u_n|_{I_j}$ to be the identity. Also,
for points $y$ outside all intervals $I_j$, define $u_n(y)=y$ for
any $n$. If it is the case that $exp_a(w)$, the sum of the
exponents of all instances of $a$ in the word $w$, is zero then
$w(x_0,u_1,u_2, \cdots ,u_{k-1})(x) \in I_0$, and hence by
construction, $w(x_0,u_1,u_2, \cdots ,u_{k-1})(x)>x$.  If
$exp_a(w) \neq 0$, then $w(x_0,u_1,u_2, \cdots ,u_{k-1})(x) \notin
I_0$, and hence $w(x_0,u_1,u_2, \cdots ,u_{k-1})(x) \neq x$.  In
either case, $w(x_0,u_1,u_2, \cdots ,u_{k-1})$ is nontrivial.
\end{proof}

We next extend Proposition \ref{prop:x0-u-substitution} to a collection of words
in $F(a,b_1, b_2, \cdots ,b_{k-1})$ of length at most $m$.  The proof of this next
proposition is analogous to that of Proposition
\ref{prop:length-at-most-m}.

\begin{proposition}
\label{prop:multiple-words}
Fix $m \in \N$, $k \geq 2$, $0 < \epsilon < \frac{1}{2}$ and let $w_1,w_2, \cdots ,w_q \in \W$.  There exists $u_1,u_2, \cdots ,u_{k-1} \in F$ with $Support(u_j) \subseteq [0, \epsilon)$ for $1 \leq j \leq k-1$ and $w_i(x_0,u_1,u_2, \cdots ,u_{k-1})|_{[0, \epsilon)}\neq 1$ for all $1 \leq i \leq q$.
\end{proposition}

\begin{proof}
Choose the interval $I_0$ as in the proof of Proposition
\ref{prop:x0-u-substitution} so that it is contained in $(0,
\epsilon)$ and its translates under $x_0^i$ are pairwise
disjoint for $i \in {\mathcal N}_m=\{-m, \cdots ,m\}$ and all contained in $(0, \epsilon)$.  Choose $q$
pairwise disjoint subintervals $J_1,J_2, \cdots ,J_q$ of $I_0$,
all having endpoints which are dyadic rationals, with the same
properties as $I_0$, that is, the translates of each $J_l$ are
pairwise disjoint under the above list of powers of $x_0$.
Following the proof of Proposition \ref{prop:x0-u-substitution},
for each $i$ with $1 \leq i \leq q$ define elements $u_{i,1},
u_{i,2},u_{i,3}, \cdots ,u_{i,k-1}$ supported in $\cup_{s \in
{\mathcal N}_m} x_0^s(J_i)$ so that $w_i(x_0,u_{i,1},
u_{i,2},u_{i,3}, \cdots ,u_{i,k-1})$ is nontrivial.

 For $1 \leq j \leq k-1$, define $u_j(x) = u_{l,j}(x)$ for $x \in \cup_{s \in
{\mathcal N}_m} x_0^s(J_l)$ over all $1 \leq l \leq q$, and
$u_j(x) = x$ for
 all other $x \in [0,1]$.  By construction, each $u_j$ is supported in $[0, \epsilon)$ for $1 \leq j \leq k-1$. Moreover, as homeomorphisms of $[0,1]$ we have
 $w_r(x_0,u_1,u_2, \cdots ,u_{k-1})|_{[0, \epsilon)} \neq 1$ for $1 \leq r \leq q$.
\end{proof}

\begin{theorem}\label{thm:girth}
Thompson's group $F$ has infinite girth.  Moreover, for any $l
\geq 3$ there is a sequence of generating sets $S_{l,n}$ of length
$l$ for $F$ so that the girth of $(F,S_{l,n})$ approaches infinity
as $n$ approaches infinity.
\end{theorem}

\begin{proof}
Fix $l \in \N$ with $l \geq 3$.  To prove the theorem we exhibit a
family of generating sets of length $l$ for $F$ so that the girth
with respect to these generating sets approaches infinity. This
proves the second statement in the theorem, which includes the
first statement in the theorem.

Choose $m \in \N$, and fix $\epsilon>0$ so that $2^{m^2} \epsilon <
\frac{1}{2}$. This choice is made initially so that later in the
argument, the supports of certain elements are disjoint from
$supp(x_1) = [\frac{1}{2},1]$.

Using Proposition \ref{prop:multiple-words} with $k=l-1$ we can
find $u_{1,m},u_{2,m}, \cdots ,u_{l-2,m} \in F$ so that for all
$w(a,b_1,b_2, \cdots ,b_{l-2}) \in {\mathcal W}_{2m^2,l-1}$ we
know that as a homeomorphism $w(x_0, u_{1,m},u_{2,m},
\cdots ,u_{l-2,m})$ is nontrivial on the interval $(0,\epsilon)$.

Consider the following generating set of length $l$ for $F$:
$$ S_{l,m}= \{ \alpha = x_0, \beta = x_0^mu_{1,m}^mx_1, \gamma_1 = u_{1,m},\ldots,\gamma_{l-2} = u_{l-2,m} \}.$$
Any freely reduced word $w$ of length at most $m$ in the free group on generators $\{\alpha, \beta, \gamma_1, \cdots, \gamma_{l-2} \}$ can be rewritten as a word $w_1$ in the free group on the generators $\{x_0, x_1, u_{1,m}, \cdots u_{l-2,m}\}$ simply replacing $\alpha$ by $x_0$, $\beta$ by $x_0^m u_{1,m}^m x_1$, and $\gamma_i$ by $u_{i,m}$, and then freely reducing, at the cost of increasing word length by a factor of at most $2m+1$. Namely,
$$w(\alpha,\beta,\gamma_1,\gamma_2, \cdots ,\gamma_{l-2})=w_1(x_0,x_1,u_{1,m},u_{2,m}, \cdots ,u_{l-2,m})$$
and since the length of $w$ is at most $m$, then the length of $w_1$
is at most $2m^2+m$.

Similarly, if we remove all instances of the letter $x_1$ from
$w_1$ and reduce the resulting word, we obtain a new word
$w_2(x_0,u_{1,m},u_{2,m}, \cdots ,u_{l-2,m})$. Since $w_2$ can be obtained from $w$ by replacing $\alpha$ by $x_0$, $\beta$ by $x_0^m u_{1,m}^m$, and $\gamma_i$ by $u_{i,m}$, and then freely reducing, $w_2$ has length at most $2m^2$.

Now observe the following fact:

\begin{fact}\label{fact}
Let $H=<x,y,d>$ be a free group, and $G=<x,y,d |d=x^my^m>$ for some positive integer $m$. Then if $w \in H$ with $w=1$ in $G$, then the length of $w$ in the free group $H$ is greater than $m$.
\end{fact}

Fact \ref{fact} implies that $w_2$ has length strictly greater than 1.  We note that $w_2$, when viewed as an element of $F$, is nontrivial, as from Proposition \ref{prop:multiple-words} we know that all nontrivial freely reduced words $w(x_0,u_{1,m},u_{2,m}, \cdots ,u_{l-2,m})$
of length at most $2m^2$ are nontrivial in $F$. Moreover, there is some $x \in (0,\epsilon)$ such that $w_2(x) \neq x$.
As
$supp(u_{i,m}) \subset [0,\epsilon)$ by construction and
$supp(x_1) \subset [\frac{1}{2},1]$, these supports are disjoint.
Since there are at most $m^2$ occurrences of the generator $x_0$ in
$w_1$ and $2^{m^2}\epsilon < \frac{1}{2}$,  we see that $w_2(x) = w_1(x)$, and hence $w_1(x) \neq x$. As $w(x) =
w_1(x)$, we have shown that $w(\alpha,\beta,\gamma_1,\gamma_2,
\cdots ,\gamma_{l-2})$ is nontrivial. Hence, the girth of
$(F,S_{l,m})$ is at least $m$.
Therefore, the girth of $(F,S_{l,m})$ approaches infinity as $m$
approaches infinity.

\end{proof}

As a direct consequence of Proposition~\ref{covergence to free},
we obtain the following corollary.

\begin{corollary}\label{cor:free-limit}
For each $l \in \N, \ l \geq 3$, the sequence of marked groups
$$G_m=(F,\{x_0,x_0^mu_{1,m}^mx_1,u_{1,m},
 \cdots ,u_{l-2,m}\})$$ converges to the free group $(F_l,\{a_1,a_2, \cdots ,a_l \})$, where $u_{1,m}, \cdots ,u_{l-2,m}$ are the
 elements constructed above in the proof of Proposition \ref{prop:multiple-words} with $k=l-1$.
\end{corollary}

\section{Non-free limits of $F$ within ${\mathcal G}_3$}
\label{sec:notfree}

The results in this section are motivated by considering several natural sequences of markings of $F$ within ${\mathcal G}_3$ of the form $\{x_0,x_1,g_n\}$ for some $g_n \in F$.  In particular, we consider the cases $g_n = x_n$, the $(n+1)$-st generator in the infinite presentation for $F$, and $g_n = x_0^n$.  However, the convergence of the resulting sequence of marked groups relies less on the actual elements chosen and more on their supports.  Hence we are able to state convergence results for more general sequences of markings of $F$.  As a corollary of Theorem \ref{thm:xn} we see that $(F,\{x_0,x_1,x_n\})$ is convergent in ${\mathcal G}_3$ and obtain a presentation for the resulting limit group; in Theorem \ref{thm:x0^n} we prove that $(F,\{x_0,x_1,x_0^n\})$ is convergent in ${\mathcal G}_3$ as well and give a presentation of the limit group.   For consistency in the notation of the marking, we identify $a=x_0$, $b=x_1$ and $c=g_n$, the additional generator in the marking.

When limits of free groups are studied, a standard construction used to create new examples of limit groups from existing examples is the generalized double.  Namely, if a group $G$ is a generalized double over a limit group $L$ (in the sense of Sela), then it is itself a limit group in this sense.  This construction is sufficient to create the entire class of limit groups (as limits of free groups), as proven by Champetier and Guirardel in the following theorem, derived from work of Sela.

\begin{theorem}[\cite{CG}, Theorem 4.6] A group is a limit group if and only if it is an iterated generalized double.
\end{theorem}

To make this precise, we define the notion of a generalized double over a limit group $L$ which is the limit of marked copies of free groups.

\begin{definition}
A generalized double over a limit group $L$ is a group $G = A*_CB$ (or $G=A*_C$) such that both vertex groups $A$ and $B$ are finitely generated and
\begin{enumerate}
\item $C$ is a nontrivial abelian group whose images under both embeddings are maximal abelian in the vertex groups, and
\item there is an epimorphism $\varphi: G \twoheadrightarrow L$ which is one-to-one in restriction to each vertex group.
\end{enumerate}
\end{definition}

While there is no analogous characterization for limits of non-free groups, we note that in each of our theorems below there is an amalgamated product (or HNN extension) of copies of Thompson's group $F$, or subgroups of $F$, over a maximal abelian subgroup which is a generalized double over the limit group obtained.  Thus the same structure emerges in our limit groups as appears in the case of limits of marked free groups.

The first example, given in Theorem \ref{thm:xn}, is a generalization of the natural sequence
of markings $\{x_0,x_1,x_n\}$ of $F$ in ${\mathcal G}_3$, that is, $a=x_0$, $b=x_1$ and $c=x_n$.

\begin{theorem}\label{thm:xn}
Let $$ G_n= \langle a, b, c | R_1=[ba^{-1}, a^{-1}ba],
R_2=[ba^{-1}, a^{-2}ba^2], c^{-1}w_n \rangle,$$ where $w_n$ is a
word in $a$ and $b$, such that viewed as a element of $F$ where
$a=x_0$ and $b=x_1$, $w_n$ has support in $[t_n,1]$, where
$\lim_{n \rightarrow \infty} t_n = 1$, and $w_n$ maps
$[\frac{3+t_n}{4},1]$ linearly to $[\frac{1+t_n}{2},1]$. Then

the sequence $(G_n, \{a,b,c\})$ converges to $(G, \{a,b,c\})$,
where
$$G= \langle a,b,c | R_1, R_2, R_3= [ca^{-1}, a^{-1}ca],
R_4=[ca^{-1}, a^{-2}ca^2], \text{and}~ [ba^{-1}, a^i c a^{-i}]
~\text{for all}~ i \in \Z  \rangle .$$
\end{theorem}

\begin{proof}
First, choose $N_0$ so that for $n \geq N_0$, $t_n \geq \frac{1}{2}$. Then for any $n\geq N_0$, $R_3$ and $R_4$ are true in $G_n$ for the following reason.
Since $w_n$ maps $[\frac{3+t_n}{4},1]$ linearly to $[\frac{1+t_n}{2},1]$, the support of $ca^{-1}$ is contained in $[0, \frac{1+t_n}{2}]$.
Now since $t_n \geq \frac{1}{2}$, $a^{-1}$ takes $[t_n,1]$ linearly to $[\frac{1+t_n}{2},1]$. As the support of $a^{-i}ca^i$ is $a^{-i}(Supp(c))$ and the support of $c$ is contained in $[t_n,1]$, we see that $Supp(a^{-i}ca^i) \subseteq a^{-i}[t_n,1] \subseteq [\frac{1+t_n}{2},1]$ for $i \geq 1$. As the supports of these two elements are disjoint, they must commute and we obtain the relations $R_3$ and $R_4$.

Next, for a relator of the form $[ba^{-1}, a^i c a^{-i}]$, choose
$N_i$ (and there are infinitely many such choices) so that
$a^i[t_{N_i},1] \subseteq [3/4,1]$. Then for any $n\geq N_i$, we see that $a^i c
a^{-i}$, as a homeomorphism in
$G_n$, has support in $[3/4,1]$ and hence commutes in $G_n$ with $ba^{-1}$, whose support lies in
$[0,3/4]$.  Therefore the relation $[ba^{-1}, a^i c a^{-i}]$ holds in $G_n$ for all $n \geq N_i$.

We follow Proposition \ref{convergence to anything}, and first prove that if $w$ is trivial in $G$
then $w$ is trivial in $G_n$ for sufficiently large $n$.  Let $w=w(a,b,c)$ be a word in $a,b,c$ and their inverses which is the
identity in $G$. Then $w$ may be expressed as a product of
conjugates of finitely many relators of $G$. But, for any finite
set of relators, there is some $N$ so that the relators all hold
in $G_n$ for $n \geq N$, and thus $w$ is trivial in $G_n$ for
all $n \geq N$.

Now suppose $w$ is a word in the letters $a, b,c$ and their
inverses which is the identity in $G_n$ for infinitely many $n$. We must show that $w$ must be trivial in $G$ as well.

We first show that the word $w$ can be rewritten in a certain constrained form in both the group $G$ and the groups $G_n$ for $n$ sufficiently large.
First, in $G$, using  relations of the form $[ba^{-1},
a^i c a^{-i}]$ for at most finitely many $i$, we may move all occurrences of $c^{\pm 1}$ to the left of
all occurrences of $b^{\pm 1}$ with the penalty of increasing the number of occurrences of the
generator $a$ in the word.

For example, suppose $w$ has a suffix of the form $w'=ba^ica^j$. Then $w'=(ba^{-1})(a^{i+1}ca^{-(i+1)})a^{j+i+1}$. Since $[ba^{-1},a^{i+1}ca^{-(i+1)}]$ is a relator in $G$, then in $G$, $w'=(a^{i+1}ca^{-(i+1)})(ba^{-1})a^{j+i+1}$. Moving through the word to the left in this manner, using a finite number of the commutator relators,
we can rewrite $w$ in this way as $w_1w_2$, where $w_1$
is a word in the generators $a$ and $c$ and $w_2$ is a word in the generators $a$ and $b$.
Choose $M$ large enough so that the finite collection of relators used in this process all hold in $G_n$ for $n > M$, and thus we can
rewrite $w$ in $G$ and in $G_n$  for $n>M$ in the desired form.

In addition, as long as we choose $M \geq N_0$  as well, then
in both $G_n$ and in $G$, the elements $a$ and $c$ satisfy the relators $R_3$ and $R_4$.
The subgroup $\langle a,c | R_3,R_4 \rangle \cong F$ and so we
can rewrite $w_1$ in an infinite
normal form using $c_0=a, c_1=c, \cdots  ,c_{1+i}= a^{-i} c a^i, \cdots$ for $i
\geq 1$. Similarly, we can rewrite $w_2$ using the relators $R_1$ and
$R_2$ in the standard infinite normal form in the letters $x_0=a, x_1=b, x_2, \cdots ,
x_{1+i}=a^{-i} b a^i, \cdots$ for $i\geq 1$, in both $G$ and $G_n$.
Suppose that $w_1=c_0^{\epsilon} w_1' c_0^{-\delta}$, where $w_1'$ has infinite
normal form in $c_i^{\pm 1}$ for $i \geq 1$, and $w_2=x_0^{\alpha} w_2'
x_0^{-\beta}$, where $w_2'$ has infinite normal form in $x_i^{\pm
1}$ for $i \geq 1$. Combine the terms $c_0^{-\delta}$ and $x_0^{\alpha}$, and if $\alpha=\delta$, these terms cancel. Otherwise, we move the combination to either the right or the left as follows. If $\alpha-\delta < 0$, rewrite $c_0^{-\delta}x_0^{\alpha}w_2'x_0^{-\beta}=x_0^{\alpha-\delta}w_2'x_0^{-\beta}$ in infinite normal form using $R_1$ and $R_2$. If $\alpha-\delta>0$, rewrite $c_0^{\epsilon} w_1' c_0^{-\delta}x_0^{\alpha}=c_0^{\epsilon} w_1' c_0^{\alpha-\delta}$ in infinite normal form in the generators $\{c_0, c_1, c_2, \dots \}$ and their inverses, using $R_3$ and $R_4$.
As a result of the application of these relators, subscripts of letters in $w_1'$ and $w_2'$ may be increased, and the final and initial exponents $\epsilon$ and $\beta$ may change as well.  Without loss of generality we then assume that there is an $M \in \N$ so that for $n>M$ we can write
$w=w_1w_2$ in both $G$ and $G_n$, where $w_1=c_0^{\epsilon}
w_1'$, $w_2= w_2' x_0^{-\beta}$, and $w_1'$ has infinite normal form
in $c_i^{\pm 1}$, for $i \geq 1$, and $w_2'$ has infinite normal form in
$x_i^{\pm 1}$, for $i \geq 1$.

Now recall that $w$ is the identity in $G_n$ for
infinitely many $n$. Then $w$ is certainly the identity for
infinitely many $n>M$, where we can write
$w=w_1w_2=c_0^{\epsilon}w_1'w_2'x_0^{-\beta}$ as described above. For any such $n$, in $G_n$, the word $w$ represents a
particular homeomorphism. As a homeomorphism, $c_1$ has support in
$[t_{n},1]$, so $c_k$ has support contained in $[t_n,1]$ as well for all $k \geq 2$.  Thus
$w_1'$ must also have support in $[t_{n},1]$. Similarly,
as the support of $x_1$ is $[\frac{1}{2},1]$ and the support of $x_k$ is contained in $[\frac{1}{2},1]$ for all $k \geq 2$,
we see that $w_2'$ has support in $[\frac{1}{2}, 1]$. But then the slope of
$w_1w_2$ near zero will be $2^{\beta-\epsilon}$; as this
homeomorphism is the identity in $G_{n}$, we must have
$\epsilon=\beta$.

For each of the infinitely many $n>M$ for which $w$ is the
identity in $G_n$, recalling that $x_0$ and $c_0$ are both equal
to the generator $a$, we may conjugate $w$ by $a^{\epsilon}$ to
obtain the word $w_1'w_2'$, which must also be the identity in $G_n$.
We claim that $w_2'$ must in fact be the empty word. For if not,
then thinking of $w_2'$ as a homeomorphism in $G_n$, there is some
$x \in (0,1)$ which  is not fixed by $w_2'$. However, for
sufficiently large $n$, $w_2'(x)$ will be outside of the support
of $w_1'$, and hence $w_1' w_2'$ will not fix $x$ in $G_n$ for
such large $n$. But if $w_2'$ is the empty word, then $w_1'$ must
be the empty word as well. But this means that in fact, for all
$n>M$, the original $w$ can be written as
$c_0^{\epsilon}x_0^{-\epsilon}$ in $G$, where $c_0=x_0=a$, and hence can be
transformed to the identity in $G$.
\end{proof}

Using the notation in the definition of the generalized double over a limit group, we remark that when
$A = B = F$ and $C = \Z$, where both inclusions of $C$ in $A$ and $B$ map $C$ to the subgroup generated by $x_0$,
then $A *_C B$ is a generalized double over the limit group $G$ obtained in Theorem \ref{thm:xn}.  We note that as the support of $x_0$
is the entire interval $[0,1]$, the subgroup generated by $x_0$ is a maximal abelian subgroup of $F$ \cite{BS2}.

In the previous example, the additional generator $c$ in $G_n$ had support
in a small neighborhood of $1$ for large $n$. Alternatively, if
we choose a sequence of additional generators to have supports in arbitrarily
small intervals close to zero but not including zero, we obtain
the following convergent sequence of marked copies of $F$.

\begin{theorem}\label{thm:jens-example}
Choose $g_n \in F$ to have support in $[r_n,s_n]\subseteq [0,1]$,
where $r_n < s_n < 2r_n$ and $\lim_{n \rightarrow \infty} r_n
=0$, and choose a word $w_n$ in $a^{\pm1},b^{\pm1}$ so that
$w_n(x_0,x_1)=g_n$. Let $$G_n=\langle a, b, c \mid R_1=[ba^{-1},
a^{-1}ba], R_2=[ba^{-1}, a^{-2}ba^2], c^{-1}w_n(a,b) \rangle.$$
Then the sequence $(G_n, \{a,b,c\})$ converges to $(G,
\{a,b,c\})$, where $$G= \langle a,b,c \mid R_1, R_2,
[a^ica^{-i},c], [a^i b a^{-i},c ] \text{ for every } ~i \in
\Z\rangle.$$
\end{theorem}

\begin{proof}
First, suppose $w=1$ in $G$. Then $w$ can be transformed to the
empty word using only a finite number of relations of $G$. There exists
some $M$ so that $s_n < 1/2$ for all $n \geq M$, and then it
follows that $x_0([r_n,s_n])$ is disjoint from $[r_n,s_n]$.
Therefore, for any $n \geq M$, the relations of the form
$[a^ica^{-i},c]$ hold in $G_n$. On the other hand, we may choose $N_i$ so that $[a^i b
a^{-i},c]=1$ holds in $G_n$ for $n>N_i$; it follows that there is an $N = \max \{M,N_i\}$ so that the finite set of relations used to transform $w$ to the empty word hold in $G_n$ for $n \geq N$, and thus $w$
is trivial in $G_n$ for $n \geq N$.

Now suppose $w=1$ in $G_n$ for infinitely many
$n$; once again we show that $w$ can be rewritten in a particular form in both $G_n$ and $G$ for sufficiently large $n$. By inserting pairs of the form $a^{-i}a^i$ adjacent to certain instances of the generator $c^{\pm 1}$ in the
word $w$ as in the proof of Theorem \ref{thm:xn} and using a finite number of the fourth type of relation given in the presentation for $G$, the word
$w$ can be rewritten in the form
$w_1w_2$ (in both $G$ and $G_n$ for sufficiently large $n$) where $w_1$ is a word in the generators $a^{\pm 1}$ and $c^{\pm 1}$, and $w_2$ is a
word in the generators $a^{\pm 1}$ and $b^{\pm 1}$.

As above, let $c_0=c$ and $c_i=a^i c a^{-i}$ for $i \geq 1$. Then the word $w_1$ can be rewritten, in any $G_n$ and in $G$, as a word in the $c_i^{\pm 1}$, at the expense of a power of $a$ at the
right, which we shift into $w_2$. So we may assume that $w_1$ is a word
in the $c_i$, and $w_2$ is a word in $a$ and $b$. For sufficiently large
$n$, as a homeomorphism in $G_n$, the element $w_1$ has slope 1 near $x=0$
and $x=1$, so $w_2$ must be supported in $[\epsilon, 1-\epsilon]$
for some $\epsilon > 0$. But for sufficiently large $n$, we know that $w_1$
is supported in $[0,\epsilon]$ as a homeomorphism in $G_n$.
Therefore, since $w$ is the identity in $G_n$ for infinitely many of these $n$, it follows that $w_2$ must be the identity, and thus can be
reduced to the empty word using $R_1$ and $R_2$. Therefore, $w_1$
is the identity in $G_n$ for infinitely many $n$. But in both $G_n$ for $n>M$ and
in $G$, the element $c_i$ commutes with $c_j$ for every $i$ and $j$. Since the
supports of these elements, viewed as homeomorphisms) are disjoint for $n>M$, $exp_{i}(w_1)$, the net exponent of all
occurrences of $c_i$, must be zero for every $i$. Thus $w_1$ can
be transformed to the empty word in $G$, and hence $w$ is the identity in $G$.
\end{proof}

Using the notation in the definition of a generalized double over a limit group, we remark that when
$A = F$, $B = \Z \wr \Z = \langle a,c | [a^{-i}ca^i,a^{-j}ca^j] \text{ for all } i,j \in \Z \rangle$ and $C = \Z$, where the inclusions of $C$ in $A$ maps $C$ to the subgroup generated by $x_0$, and the inclusion of $C$ in $B$ maps $C$ to the subgroup generated by $a$, then $A *_C B$ is a generalized double over the limit group $G$ obtained in Theorem \ref{thm:jens-example}.

In the list of motivating ``natural" sequences of markings, it remains to consider the case where
the additional generator $g_n$ of $G_n$ is taken to be $x_0^n$.

\begin{theorem}\label{thm:x0^n}
Let $$G_n= \langle a, b, c | R_1=[ba^{-1}, a^{-1}ba],
R_2=[ba^{-1}, a^{-2}ba^2], c^{-1}a^{n}\rangle.$$ Then the sequence
$(G_n, \{a,b,c\})$ converges to $(G, \{a,b,c\})$, where $$G=
\langle a,b,c | R_1, R_2,[c,a], [
(a^{-j}ba^j)a^{-1},c^{-i}(a^{-k}ba^k) c^i], i \geq 1, j \geq 0 , k
\geq 0 \rangle .$$
\end{theorem}

\begin{proof}
Note first that in both $G$ and $G_n$, $a$ and $b$ generate a
subgroup isomorphic to $F$, and that the relator $[c,a]$ holds in $G_n$ for all $n$ as $c=a^n$.
As above, we identify $x_0$ with
$a$, $x_1$ with $b$, and recall that $x_{i+1} = a^{-i}ba^i$; making these substitutions into the
final relator in the presentation of $G$ with $c=a^n$, we see that this relator, viewed in $G_n$, claims that $x_{ni+k+1}$ commutes with $x_{j+1}x_0^{-1}$.
As the support of $x_{ni+k+1}$  is easily seen to be $[1- \frac{1}{2^{ni+k+1}},1]$ and the support of $x_{j+1}x_0^{-1}$ is $[0,1- \frac{1}{2^{j+2}}]$,
a relator of this form is satisfied in $G_n$ as long as $j+2\leq ni+k+1$, that is, $n \geq \frac{j-k+1}{i}$.

We note for later use that the last two types type of relators
in $G$, when combined, yield relations of the form $$x_{j+1} (c^{-i}
x_{k+2}c^i)=(c^{-i}x_{k+1}c^i)x_{j+1},$$ for $j \geq 0, i \geq 1,
k \geq 0$. Furthermore, if $n \geq (j-k+1)/i$, then this relation
holds in $G_n$ as well.

Suppose $w$ is a word in $a,b,c$ and their inverses which is the identity in
$G$. Then it may be expressed as a product of conjugates of
finitely many relators of $G$. But, for any finite set of
relators, there is some $M$ so that the relators all hold in $G_n$
for $n \geq M$, so $w$ is also the identity in $G_n$ for all $n
\geq M$.

Given any word $w$ in $a,b,c$ and their inverses, $w$ can be expressed, in both
$G$ and $G_n$ for any $n$, as
$$c^{n_1}w_1c^{n_2}w_2\ldots c^{n_k}w_k,$$ where for each $i$, $w_i$ is a word
in $a^{\pm 1}$ and $b^{\pm 1}$ and $n_j \neq 0$ for $2 \leq j \leq k$. As usual, using the notation $x_0=a, x_1=b,
x_{1+i}=a^{-i} b a^i$, we may assume for each $i$ that $w_i$ is a word in the
standard infinite normal form for $F$. Next, since $x_0$ commutes
with $c$ in both $G$ and $G_n$, and then also using relators
involving just the $x_j$, that is, the standard relators in $F$ which are a consequence of $R_1$ and $R_2$, we may assume (changing the $w_i$'s
without renaming) that $w$ is of the form
$$x_0^ac^{n_1}w_1c^{n_2}w_2\ldots c^{n_k}w_kx_0^{-b},$$ where $a$
and $b$ are positive integers and $w_i$ is a word in infinite
normal form without $x_0^{\pm 1}$, in both $G$ and in $G_n$ for any $n$.

Now suppose that $w$ is the identity in $G_n$ for infinitely many
$n$. Then it follows that the total exponent sum of $x_0$ must be
zero for those indices $n$, in other words, $(a-b)+(n_1+n_2+\cdots
+n_k)n=0$ for each of those indices $n$. Therefore $a=b$ and
$n_1+n_2+\cdots +n_k=0$. But $w=id$ in $G_n$ if and only if
$x_0^{-a}wx_0^a=id$, so we know that
$$w'=c^{n_1}w_1c^{n_2}w_2\ldots c^{n_k}w_k,$$ where $w_i$ is a word
in infinite normal form without $x_0^{\pm 1}$, is the identity in
$G_n$ for infinitely many $n$. Moreover, conjugating if necessary,
we may assume that $n_i \neq 0$ (for $i \neq 1$) and $w_i$ is not
the empty word, for all $i$. Now since $n_1+n_2+\cdots +n_k=0$, we
may rewrite $w'$ as:
$$w'=(c^{n_1}w_1c^{-n_1})(c^{n_1+n_2}w_2c^{-(n_1+n_2)})\ldots (c^{n_1+\cdots n_{k-1}}w_{k-1}c^{-(n_1+\cdots n_{k-1})})(w_k).$$

We claim that the word $w'$ must also be the identity in $G$. For if not,
suppose that amongst the words of this form which are the identity
in $G_n$ for infinitely many $n$ and are not the identity in $G$,
the word $w'$ has $k$ minimal. Next, let $t=max(n_1, n_1+n_2, \ldots,
n_1+n_2+\cdots+n_{k-1})$, and since $w'$ is the identity (in
$G_n$ or $G$) if and only if $c^{-t}w'c^t$ is the identity, we may
assume $w'$ is of the form:
$$(c^{-m_1}w_1c^{m_1})(c^{-m_2}w_2c^{m_2})\cdots
(c^{-m_k}w_kc^{m_k}),$$ where $m_i\geq 0$ for all $i$, and at least
one $m_i=0$. For each value of $i$ where $m_i=0$, the subword
$c^{-m_i}w_ic^{m_i}=w_i=p_iq_i$, where $p_i$ (resp.
$q_i$) is a positive word (resp. a negative word) in normal
form involving no $x_0^{\pm1}$. Thus, in $G$, using finitely many
relators of the form $$x_{j} (c^{-i}
x_{s+1}c^i)=(c^{-i}x_sc^i)x_j$$ for $j \geq 1, i \geq 1$, we can
move $p_i$ to the left and $q_i$ to the right of the expression, until eventually
$w$ can be written (reindexing) as $p (\Pi_{i=1}^l
c^{-m_i}w_ic^{m_i})q$, where $l<k$ and $p$ (resp. $q$) is a
positive (resp. negative) word in infinite normal form.
Since this can be done in $G$ using only finitely many relations, for
sufficiently large $n$ it can be done in $G_n$ as well, so we may
assume it can be done in $G_n$ for infinitely many $n$. Now notice
that by choosing the minimal value of the index $n$ to be perhaps even
larger, we can ensure that once we replace $c$ by $x_0^n$ in
$G_n$, the product $\Pi_{i=1}^l c^{-m_i}w_ic^{m_i}$, when written in the standard infinite
normal form, involves only the generators $x_j$ with $j$ much
larger that the subscripts of the generators in $p$ or in $q$. But
since $w$ is the identity for infinitely many $n$, it follows that
$p=q^{-1}$. Hence, $\Pi_{i=1}^l c^{-m_i}w_ic^{m_i}$, with $l <k$,
is the identity in $G_n$ for infinitely many $n$, but is not the
identity in $G$, which is a contradiction, as we assumed that $k$ was minimal.
\end{proof}

Using the notation in the definition of the generalized double over a limit group, we remark that when
$A = F$, and $C = \Z$, where both inclusions of $C$ in $A$ map $C$ to the subgroup generated by $x_0$, then $A *_C$ is a generalized double over the limit group $G$ obtained in Theorem \ref{thm:x0^n}.

\end{document}